\documentclass[11pt]{article}
\usepackage{amssymb,amsmath,upgreek,amsthm,hyperref}
\usepackage{graphicx}
\usepackage{color,centernot}
\usepackage{fullpage}

\usepackage{amsfonts,amssymb}
\usepackage{graphicx}
\usepackage{amsmath,amssymb}
\usepackage{epsf}
\usepackage{epsfig}
\usepackage{graphics}
\usepackage{cite}
\usepackage{color}
\usepackage{enumerate}
\usepackage{amssymb,amsmath}
\usepackage{graphicx,enumerate,color}
\newcommand{\NP}{\ensuremath{\mathsf{NP}}}

\newcommand{\book}[3]{{#1}, {\it #2}, #3.}
\newcommand{\article}[6]{{#1}, {#2}, {\it #3} {#4} (#5) #6.}
\newcommand{\incollection}[3]{{#1}, {#2}, {#3}.}

\newtheorem{theorem}{Theorem}[section]
\newtheorem{lemma}[theorem]{Lemma}

\newtheorem{corollary}[theorem]{Corollary}
\newtheorem{proposition}[theorem]{Proposition}

\newtheorem{definition}{Definition}

\newtheorem{problem}{Problem}

\newtheorem{remark}{Remark}

\newtheorem{example}{Example}

\title{$\mbox{Total Domishold Graphs: a Generalization of Threshold Graphs,}$\\ $\mbox{with Connections to Threshold Hypergraphs}$\thanks
{This work is supported in part by the
Slovenian Research Agency, research program P$1$--$0285$ and research projects
J$1$--$4010$, J$1$--$4021$ and N$1$--$0011$:
GReGAS, supported in part by the European Science Foundation. Some results from this paper
appeared as extended abstract in the proceedings of the conference WG $2013$~\cite{ChiMil13}.}}

\author{Nina Chiarelli\\
\small University of Primorska, UP FAMNIT, Glagolja\v ska 8, SI6000 Koper, Slovenia\\
\small \texttt{nina.chiarelli@student.upr.si}\\
\and
Martin Milani\v c\thanks{Corresponding author.}\\
\small University of Primorska, UP IAM, Muzejski trg 2, SI6000 Koper, Slovenia\\
\small University of Primorska, UP FAMNIT, Glagolja\v ska 8, SI6000 Koper, Slovenia\\
\small \texttt{martin.milanic@upr.si}
}

\date{\today}
\begin{document}
\maketitle

\begin{abstract}
A total dominating set in a graph is a set of vertices such that every vertex
of the graph has a neighbor in the set. We introduce and study graphs that admit
non-negative real weights associated to their vertices such that a set of vertices is a total dominating set if and only if the sum
of the corresponding weights exceeds a certain threshold.
We show that these  graphs, which we call total domishold graphs,
form a non-hereditary class of graphs properly containing the classes of threshold graphs and the complements of domishold graphs, and
are closely related to threshold Boolean functions and threshold hypergraphs.
We present a polynomial time recognition algorithm of total domishold graphs, and characterize
graphs in which the above property holds in a hereditary sense. Our characterization is obtained by studying a new family of hypergraphs,
defined similarly as the Sperner hypergraphs, which may be of independent interest.

\medskip


\smallskip {\bf Math.~Subj.~Class.}~(2010): 05C69
\end{abstract}

\section{Introduction and Background}

A possible approach for dealing with the intractability of a given decision or optimization problem is to identify restrictions on input instances under which the problem can still be solved efficiently. One generic framework for describing a kind of such restrictions
in case of graph problems is the following: Given a graph $G$, does $G$ admit non-negative integer
weights on its vertices (or edges, depending on the problem) and a set $T$ of integers such that a subset $X$ of its vertices (or edges) has property $P$
if and only if the sum of the weights of elements of $X$ belongs to $T$? Property $P$ can denote any of the desired substructures we are looking for, such as matchings, cliques, stable sets, dominating sets, etc. If weights as above exist and  are given with the graph, and the set $T$ is given by
a membership oracle, then a dynamic programming algorithm can be employed to
find a subset with property $P$ of either maximum or minimum cost (according to a given cost function on the vertices/edges)
in $O(nM)$ time and  with $M$ calls of the membership oracle, where $n$ is the number of vertices (or edges) of $G$ and $M$ is a given upper bound for $T$~\cite{MOR}.

The advantages of the above framework depend both on the choice of property $P$ and on the constraints (if any) imposed on the structure of the set $T$. For example, if $P$ denotes the property of being a stable (independent) set,
and set $T$ is restricted to be an interval unbounded from below---equivalently, $T$ is of the form
of the form $[0,t]$ for some non-negative integer $t$---, we obtain the class of {\it threshold graphs}~\cite{CH77}, which
is very well understood and admits several characterizations and linear time algorithms for recognition and for several optimization problems~\cite{MP}. If $P$ denotes the property of being a dominating set and $T$ is an interval
unbounded from above, that is, it is of the form $T = [t, \infty)$ for some non-negative integer $t$,
we obtain the class of {\it domishold graphs}~\cite{BH78}, which enjoy similar properties as the threshold graphs. On the other hand, if $P$ is the property of being a {\it maximal} stable set and $T$ is restricted to consist of a single number, we obtain the class of {\it equistable graphs}~\cite{Pay}, for which the recognition complexity is open (see, e.g.,\cite{LMT}), no structural characterization is known, and several \NP-hard optimization problems remain intractable on this class. For instance, it is shown in~\cite{MOR} that the maximum independent set and the minimum independent dominating set problems are APX-hard in the class of equistable graphs.

As the above examples show, the resulting class of graphs can be either {\it hereditary} (that is, closed under vertex deletion)---as in the case of threshold or domishold graphs---, or not---as in the case of equistable graphs. When the resulting graph class is not hereditary, it is natural to consider the hereditary version of the property, in which the requirement (the existence of weights and the set $T$) is extended to all induced subgraphs of the given graph.

In this paper, we introduce and study the case when $P$ is the property of being a total dominating set and $T$ is an interval unbounded from above. Given a graph $G = (V,E)$, a {\it total dominating set} (a TD set, for short) is a subset $D$ of the vertices of $G$ such that every vertex of $G$ has a neighbor in $D$. This notion has been extensively studied in the literature, see, e.g.~the monographs~\cite{HHS1-98,HHS2-98,Henning-Yeo} and the survey paper~\cite{Henning}.

\begin{definition}
A graph $G=(V,E)$ is said to be {\em total domishold} (TD for short)
if there exists a pair $(w,t)$ where $w:V\to \mathbb{R}_+$ is a weight function and $t\in \mathbb{R}_+$ is a threshold such that
for every subset $S\subseteq V$, $w(S):= \sum_{x\in S}w(x)\ge t$ if and only if $S$ is a total dominating set in $G$.
A pair $(w,t)$ as above will be referred to as a {\em total domishold structure} of $G$.
\end{definition}

Notice that the above definition allows $G$ to have isolated vertices. Every graph with an isolated vertex is total domishold, even though it does not have any TD sets.

\begin{example}
The complete graph of order $n$ is total domishold. Indeed, a subset $S\subseteq V(K_n)$ is a total dominating set of $K_n$ if and only if $S$ is of size at least two, and consequently the pair $(w,2)$ where $w(x) = 1$ for all $x\in V(K_n)$ is a total domishold structure of $K_n$.
On the other hand, the $4$-cycle $C_4$ is not a total domishold graph (cf.~Theorem~\ref{prop:forbidden-induced-subgraphs} in Section~\ref{sec:HTD-partial-characterizations}).
\end{example}

It is easy to see that adding a new vertex to the $4$-cycle and connecting it to exactly one vertex of the cycle results in a total domishold graph.
Therefore, contrary to the classes of threshold and domishold graphs, the class of total domishold graphs is not hereditary.
This motivates the following definition:

\begin{definition}
A graph $G$ is said to be  {\em hereditary total domishold} ({HTD} for short) if every induced subgraph of it is total domishold.
\end{definition}

We initiate the study of TD and HTD graphs. We identify several operations preserving the class of TD graphs, which, together with results from the literature~\cite{BH78, CH77}, imply that the class of HTD graphs properly contain the classes of threshold graphs and the complements of domishold graphs (Section~\ref{sec:properties}). We exhibit close relationships between total domishold graphs, threshold Boolean functions and threshold hypergraphs (in Section~\ref{sec:TD-graphs-BFs-hypergraphs}), and use them to develop a polynomial time recognition algorithm for total domishold graphs, a polynomial time algorithm for the total dominating set problem in the class of total domishold graphs, and a $2$-approximation algorithm for the dominating set problem in the class of total domishold graphs (in Section~\ref{sec:algorithmic-aspects}).
As our main result, we completely characterize the HTD graphs (in Section~\ref{sec:HTD-partial-characterizations}), both in terms of forbidden induced subgraphs, and in terms of properties of certain Boolean functions derived from induced subgraphs of the graph.
We conclude the paper with  some  open problems in Section~\ref{sec:conclusion}.

\section{Preliminaries and Notation}

In this section we recall some useful definitions and results.
For a positive integer $n$, we will use the notation $[n]$ for the set $\{1,\ldots, n\}$.

\medskip
\noindent {\bf Graphs and graph classes.}
A {\it clique} in a graph is a set of pairwise adjacent vertices, and an {\it independent} (or {\it stable}) set is a set of pairwise non-adjacent vertices.
A graph $G$ is {\it chordal} if it does not contain any induced cycle of order at least $4$,
{\it split} if its vertex set can be partitioned into a clique and an independent set, and
{\it $(1,2)$-polar} if it admits a partition of its vertex set into two (possibly empty) parts $K$ and $L$ such that
$K$ is a clique and $L$ induces a subgraph of maximum degree at most~$1$.
For a set of graphs ${\cal F}$, a graph $G$ is said to be {\it ${\cal F}$-free} (or just {\it $F$-free} if ${\cal F} = \{F\}$),
if it does not contain any induced subgraph isomorphic to a member of ${\cal F}$.
Every member of ${\cal F}$ is said to be a {\it forbidden induced subgraph} for the (hereditary) set of ${\cal F}$-free graphs.
The neighborhood of a vertex $v$ in a graph will be denoted by $N_G(v)$, and its closed neighborhood by $N_G[v] := N_G(v)\cup \{v\}$,
omitting the subscript $G$ if the graph is clear from the context.
For a subset $U\subseteq V(G)$ of vertices, we denote by $N_G(U)$ (or just by $N(U)$) the set $\{w\in V(G)\setminus U\mid (\exists u\in U)(uw\in E(G))\}$.
A vertex in a graph $G$ is {\it universal} if it is adjacent to every other vertex in $G$ and {\it isolated} if it is of degree~$0$.
By $G+H$ we will denote the disjoint union of graphs $G$ and $H$.
The {\it join} of graphs $G$ and $H$ is the graph obtained from the disjoint union $G+H$ by adding all edges of the form $\{uv\mid u\in V(G),~v\in V(H)\}$.
For a graph $G$, we denote by $2G$ the disjoint union of two copies of $G$. We denote by $K_n$, $P_n$ and $C_n$ the complete graph, the path and the cycle on $n$ vertices.

The following characterization of threshold graphs due to Chv\'atal and Hammer will be used in some of our proofs.

\begin{sloppypar}
\begin{theorem}[Chv\'atal and Hammer~\cite{CH77}]\label{thm:ChvatalHammer}
For every graph $G$ the following conditions are equivalent:
\begin{enumerate}[(a)]
  \item\label{itemCHa} $G$ is threshold.
  \item\label{itemCHb} $G$ has no induced subgraph isomorphic to $2K_{2},C_{4}$ or $P_{4}$.
  \item\label{itemCHd} There is a partition of $V$ into two disjoint sets $I,K$ and an ordering $u_{1},u_{2},\ldots,u_k$
of vertices in $I$ such that no two vertices in $I$ are adjacent, every two vertices in $K$ are
adjacent, and $N(u_{1})\subseteq N(u_{2})\subseteq\ldots\subseteq N(u_k)$.
Without loss of generality, we will assume that $I$ is a maximal independent set of $G$, that is, that $\cup_{u\in I}N(u) = K$.
\end{enumerate}
\end{theorem}
\end{sloppypar}

\medskip
\noindent {\bf Boolean functions.}
Let $n$ be  positive integer. Given two vectors $x,y\in \{0,1\}^n$, we write $x\le y$ if $x_i\le y_i$ for all $i\in [n] = \{1,\ldots, n\}$.
A Boolean function $f:\{0,1\}^n\to\{0,1\}$ is {\it positive} (or: {\it monotone}) if $f(x)\le f(y)$ holds for
every two vectors $x,y\in \{0,1\}^n$ such that $x\le y$.
A positive Boolean function $f$ is said to be {\it threshold} if there exist non-negative real weights $w=(w_1,\ldots, w_n)$ and a non-negative real number $t$ such that
for every $x\in \{0,1\}^n$, $f(x) = 0$ if and only if $\sum_{i = 1}^nw_ix_i\le t$~(see, e.g.,~\cite{CH11}).
Such a pair $(w,t)$ is called a {\it separating structure} of $f$. Every threshold positive Boolean function admits
an integral separating structure~\cite{CH11}.

Threshold Boolean functions have been characterized by Chow~\cite{Chow} and Elgot~\cite{Elgot}, as follows. For $k\ge 2$, a positive Boolean function $f: \{0,1\}^n\to\{0,1\}$ is said to be {\it $k$-summable} if, for some $r \in \{2,\ldots,k\}$,  there exist $r$ (not necessarily distinct) false points of $f$, say,
$x^1, x^2,\ldots, x^r$, and $r$ (not necessarily distinct) true points of $f$, say $y^1, y^2,\ldots, y^r$, such that
$\sum_{i=1}^r x^i = \sum_{i=1}^r y^i$. (A
{\it false point} of $f$ is an input vector $x\in \{0,1\}^n$ such that $f(x) = 0$; a {\it true point} is defined analogously.) Function $f$ is said to be {\it $k$-asummable} if it is not $k$-summable, and it is {\it asummable} if it is $k$-asummable for all $k \geq 2$.

\begin{theorem}[Chow~\cite{Chow}, Elgot~\cite{Elgot}, see also Theorem 9.14 in~\cite{CH11}]\label{thm:BF:characterization}
A positive Boolean function $f$ is threshold if and only if it is asummable.
\end{theorem}

The problem of determining whether a positive Boolean function
given by its complete DNF is threshold is solvable in polynomial time, using dualization and linear programming.
This result is summarized in the following theorem.

\begin{sloppypar}
\begin{theorem}[Peled and Simeone~\cite{PS85}, see also Theorem 9.16 in~\cite{CH11}]\label{thm:threshold-Bfs}
~\\There exists a polynomial time
algorithm that determines, given the complete DNF of a positive
Boolean function $f(x_1,\ldots, x_n)$,
whether $f$ is threshold. If this is the case, the algorithm also computes an integral separating structure of~$f$.
\end{theorem}
\end{sloppypar}

The existence of a {\it purely combinatorial} polynomial time algorithm for recognizing threshold positive Boolean functions given by their complete DNF is an
open problem~\cite{CH11}. (By a ``purely combinatorial algorithm'' we mean an algorithm that does not rely on solving a linear program.)
Notice that a combinatorial algorithm with exponential running time follows from Theorem~\ref{thm:BF:characterization}.

For further definitions related to Boolean functions, we refer to the monograph by Crama and Hammer~\cite{CH11}.

\medskip
\noindent {\bf Hypergraphs.}
A {\it hypergraph} is a pair $\cal H = (V,E)$ where $\cal V$ is a finite set of {\em vertices} and $\cal E$ is a finite set of subsets of $\cal V$, called {\em (hyper)edges}~\cite{Berge}. For our purposes, we will allow multiple edges (unless specified otherwise), that  is, $\cal E$ can be a multiset.
A hypergraph $\cal H = (V,E)$ is {\em threshold} if there exist a weight function $w:{\cal V}\to \mathbb{R}_+$ and a threshold $t\in \mathbb{R}_+$
such that for all subsets $X\subseteq {\cal V}$, it holds that $w(X)\le t$ if and only if $X$ contains no edge of $\cal H$~\cite{Gol}.
Such a pair $(w,t)$ is called a {\it separating structure} of $\cal H$.

To every hypergraph $\cal H = (V,E)$, we can naturally associate a positive Boolean function $f_{\cal H}:\{0,1\}^{\cal V}\to\{0,1\}$, defined by the positive DNF expression
$$f_{\cal H}(x) = \bigvee_{e\in {\cal E}} \bigwedge_{u\in e}x_u$$ for all $x\in \{0,1\}^{\cal V}$.
Conversely, to every positive Boolean function $f:\{0,1\}^n\to \{0,1\}$ given by a positive  DNF $\phi = \bigvee_{j = 1}^m \bigwedge_{i\in C_j}x_i$, we can associate a hypergraph ${\cal H}(\phi) = {\cal (V,E)}$ as follows: ${\cal V} = [n]$ and ${\cal E} = \{C_1,\ldots, C_m\}$.
It follows directly from the definitions that the thresholdness of hypergraphs and of Boolean functions are related as follows.

\begin{proposition}\label{prop:relation}
A hypergraph $\cal H = (V,E)$ is threshold if and only if the positive Boolean function $f_{\cal H}$ is threshold.
A positive Boolean function given by a positive DNF $\phi = \bigvee_{j = 1}^m \bigwedge_{i\in C_j}x_i$ is threshold if and only if
the hypergraph ${\cal H}(\phi)$ is threshold.
\end{proposition}

Moreover, reformulating the characterization of threshold positive Boolean functions given by Theorem~\ref{thm:BF:characterization} in the language of hypergraphs, we obtain the following characterization of threshold hypergraphs.

\begin{theorem}[Chow~\cite{Chow}, Elgot~\cite{Elgot}]\label{thm:characterization}
A hypergraph $\cal H = (V,E)$ is {\em not} threshold if and only if there exists an integer $n$ with $n\ge 2$
and $n$ (not necessarily distinct) subsets $A_1,\ldots, A_n$ of $V$, each containing an edge of $\cal H$,
and $n$ (not necessarily distinct) subsets $B_1,\ldots, B_n$ of $V$, each containing no edge of $\cal H$, such that
for every vertex $v\in V$,
\begin{equation}\label{degree}
|\{i\mid v\in A_i\}| = |\{i\mid v\in B_i\}|\,.
\end{equation}
\end{theorem}

\section{Basic Properties}\label{sec:properties}

In this section, we establish some basic properties of TD graphs.

\begin{proposition}\label{prop:isolated-vertex}
Every graph with an isolated vertex is TD.
\end{proposition}

\begin{proof}
If $G$ has an isolated vertex, then $G$ does not have any TD sets, and hence
the pair $(w,|V(G)|+1)$, where $w(x) = 1$ for all $x\in V(G)$ is a total domishold structure of~$G$.
\end{proof}

As shown by the $4$-cycle, TD graphs are not closed under join. However, they are closed under join with $K_1$, that is, under adding a universal vertex. This is stated formally in
Proposition~\ref{prop:dominating-vertex} below and proved using the following auxiliary lemma.

\begin{lemma}\label{lem:positive-weights}
Every TD graph admits a
total domishold structure
in which all weights are positive.
\end{lemma}

\begin{sloppypar}
\begin{proof}
Let $(w,t)$ be a total domishold structure of a TD graph $G=(V,E)$.
The value of
$$\delta = t-\max\{w(S)\mid S\in {\cal P}(V)\setminus {\cal T}\}\,,$$
where
${\cal P}(V)$ denotes the power set of $V$ and
${\cal T}$ denotes the set of all total dominating sets of $G$,
is well defined and positive.
Let $w':V\to\mathbb{R}_+\setminus \{0\}$ and $t'\in \mathbb{R}$ be defined as:
$w'(x) = |V|w(x)+\delta/2$ for all $x\in V$, and $t' = |V|t$.
We claim that $(w',t')$ is a total domishold structure of $G$.
On the one hand, if $S\in {\cal T}$, then
$w'(S) = |V|w(S)+\delta|S|/2\ge |V|t = t'$.
On the other hand, if
$S\in {\cal P}(V)\setminus {\cal T}$, then
$w(S)+\delta/2<t$ and consequently
\hbox{$w'(S) = |V|w(S)+\delta|S|/2 \le |V|(w(S)+\delta/2)<|V|t = t'$.}\end{proof}
\end{sloppypar}

\begin{proposition}\label{prop:dominating-vertex}
Let $G$ be a graph, and let $G'$ be the graph obtained from $G$ by adding to it a vertex adjacent to all vertices of $G$.
Then, $G$ is TD if and only if $G'$ is TD.
\end{proposition}

\begin{proof}
The proof will follow from the observation that the sets ${\cal T}$ and ${\cal T'}$ of
total dominating sets of $G$ and $G'$, respectively, are related as follows:
$${\cal T'} = {\cal T}\cup \{\{v\}\cup S\mid \emptyset \neq S\subseteq V(G)\}\,,$$
where $v$ is the added vertex.

Suppose first that $G$ is TD. By Lemma~\ref{lem:positive-weights}, $G$ admits a total domishold structure $(w,t)$
with $w(x)>0$ for all $x\in V(G)$.
Let
$w':V(G')\to \mathbb{R}_+$ be defined as follows:
\begin{itemize}
  \item for all $x\in V(G)$, let $w'(x) = w(x)$;
  \item let $w'(v) = t - \min\{w(x)\mid x\in V(G)\}$.
\end{itemize}
We claim that $(w',t)$ is a total domishold structure of $G'$.
Indeed,
if $S\in {\cal T}'$ then we consider two cases.
If $v\not\in S$, then $S\in {\cal T}$ and $w'(S) = w(S)\ge t$.
If $v\in S$, then $\{x,v\}\subseteq S$ for some $x\in V(G)$, and hence
$w'(S) \ge w'(x)+w'(v) = w(x) + t - \min\{w(z)\mid z\in V(G)\}\ge t$.
Similarly, if $w'(S)\ge t$, we consider two cases.
If $v\not\in S$, then $w(S)\ge t$ and therefore $S\in {\cal T}\subseteq {\cal T'}$.
If $v\in S$, then $S\cap V(G)\neq \emptyset$ (since otherwise we would have $w'(S) = w'(v) < t$ by the positivity of $w$),
and thus $S\in {\cal T}'$.

The other direction is straightforward. Since ${\cal T'}\cap {\cal P}(V(G)) = {\cal T}$, any pair $(w,t)$ such that
$(w',t)$ is a total domishold structure of $G'$ and $w$ is the restriction of $w'$ to $V(G)$,
is a total domishold structure of $G$.
\end{proof}

\begin{corollary}\label{cor:threshold}
Every threshold graph is HTD.
\end{corollary}

\begin{proof}
The characterization of threshold graphs due given by Theorem~\ref{thm:ChvatalHammer}(c)
implies that every threshold graph contains either an isolated vertex or a universal vertex.
Therefore, an induction on the number of vertices together with
Propositions~\ref{prop:isolated-vertex} and~\ref{prop:dominating-vertex}
shows that every threshold graph is TD.
Since the class of threshold graphs is hereditary~\cite{CH77},
every threshold graph is HTD.
\end{proof}

In general, TD graphs are not closed under disjoint union: the path $P_3$ is TD, but the graph $2P_3$ is not (cf.~Theorem~\ref{prop:forbidden-induced-subgraphs} in Section~\ref{sec:HTD-partial-characterizations}).
However, they are closed under adding a (TD) graph with a {\it unique} (inclusion-wise) minimal TD set.

\begin{proposition}\label{prop:unique-TD_disjoint-union}
Let $G$ and $H$ be graphs such that $H$ has a unique minimal TD set.
Then, $G+H$ is TD if and only if $G$ is TD.
\end{proposition}

\begin{proof}
Let $T$ be the unique minimal TD set in $H$. Then,
the sets ${\cal T}$ and ${\cal T'}$ of
total dominating sets of $G$ and $G':= G+H$ are related as follows:
$${\cal T'} = \{S\cup T'\mid S\in {\cal T} \textrm{ and } T\subseteq T'\subseteq V(H)\}\,.$$

Suppose that $G$ is TD, with a total domishold structure $(w,t)$.
Let $N = w(V(G))$ and define $w':V(G') \to \mathbb{R}_+$ as
$$w'(x) = \left\{
            \begin{array}{ll}
              w(x), & \hbox{if $x\in V(G)$;} \\
              N, & \hbox{if $x\in T$;}\\
              0, & \hbox{otherwise.}
            \end{array}
          \right.
$$
It is not hard to verify that the pair $(w',t+|T|N)$ a total domishold structure of $G'$.

Conversely, if $(w',t')$ is a total domishold structure of $G'$, then
$(w,t'-w'(T))$, where $w$ is the restriction of $w'$ to $V(G)$, is a total domishold structure of $G$.
\end{proof}

\begin{corollary}\label{cor:isolated-edge}
Let $G$ be a graph, and let $G' = G+K_2$.
Then, $G$ is TD if and only if $G'$ is TD.
\end{corollary}

A graph $G$ is said to be {\it co-domishold} if its complement is domishold.
Since threshold graphs are exactly the domishold co-domishold graphs~\cite{BH78, CH77},
the following result generalizes Corollary~\ref{cor:threshold}.

\begin{corollary}\label{cor:co-domishold}
Every co-domishold graph is HTD.
\end{corollary}

\begin{proof}
This can be proved similarly as Corollary~\ref{cor:threshold}, using
Corollary~\ref{cor:isolated-edge} in addition to Propositions~\ref{prop:isolated-vertex} and~\ref{prop:dominating-vertex},
and the facts that:
(1) the class of co-domishold graphs is hereditary (this is because the class of domishold graph is hereditary~\cite{BH78});
(2) every co-domishold graph contains either an isolated vertex, a universal vertex, or a connected component isomorphic to $K_2$~\cite{BH78}.
\end{proof}

Note that not every HTD graph is co-domishold. For example, the $4$-vertex path $P_4$
is easily verified to be HTD but it is not domishold~\cite{BH78}, and hence (since it is self-complementary) also not co-domishold.

As observed in the introduction, the set of TD graphs is not hereditary.
We conclude this section by strengthening this observation,
showing that the set of TD graphs is not contained in any nontrivial hereditary class of graphs
(even if we disallow graphs with isolated vertices).

\begin{proposition}\label{prop:non-hereditary}
For every graph $G$ there exists a TD graph $G'$ without isolated vertices such that $G$ is an induced subgraph of $G'$.
\end{proposition}

\begin{proof}
Let $G$ be a graph. First, add to $G=(V,E)$ a new vertex, say $v$, and connect $v$ only to isolated vertices of $G$.
Second, add a new private neighbor to every vertex of the resulting graph. Denoting by $G'$ the obtained graph,
it is clear that $G$ is an induced subgraph of $G'$. By construction,
the set $V\cup \{v\}$ is the unique minimal total dominating set in $G'$. Therefore,
the pair $(w,t)$, where $w:V(G')\to\mathbb{R}_+$ is given by $w(x) = 1$
if  $x\in V\cup \{v\}$ and $w(x)  = 0$, otherwise,
and $t = |V|+1$, is a total domishold structure of $G'$.
\end{proof}

\section{Connections to Boolean Functions and Hypergraphs}\label{sec:TD-graphs-BFs-hypergraphs}

In this section, we exhibit close relationships between total domishold graphs, threshold Boolean functions and threshold hypergraphs.
We start with a characterization of TD graphs in terms of the thresholdness of a derived Boolean function.

Given a set $V$ and a binary vector $x\in \{0,1\}^V$, the {\it support set} of a vector $x\in \{0,1\}^V$ is the set $S(x) = \{v\in V\mid x_v = 1\}$. Also, by $\overline x$ we denote the vector $\overline x\in \{0,1\}^V$ given by $(\overline x)_i = 1-x_i$ for all $i\in V$.
Given a graph $G = (V,E)$, its {\it neighborhood function} is the positive Boolean function $f_G:\{0,1\}^{V}\to\{0,1\}$
that takes value $1$ precisely on vectors $x\in \{0,1\}^V$ whose support set $S(x)$
contains the neighborhood of some vertex of $G$. Formally, $$f_G(x) = \bigvee_{v\in V} \bigwedge_{u\in N(v)}x_u$$
for every vector $x\in \{0,1\}^{V}$. (If $N(v) = \emptyset$ then $\bigwedge_{u\in N(v)}x_u = 1$.)

\begin{proposition}\label{lem:reduction}
A graph $G = (V,E)$ with $V = \{v_1,\ldots, v_n\}$
is total domishold if and only if its neighborhood function $f_G$ is threshold.
Moreover, if $(w_1,\ldots, w_n,t)$ is an integral separating structure of $f_G$, then
$(w;\sum_{i = 1}^nw_i-t)$ with $w(v_i) = w_i$ for all $i\in [n]$ is a total domishold structure of $G$.
\end{proposition}

\begin{sloppypar}
\begin{proof}
First, recall that a positive Boolean function $f(x_1,\ldots, x_n)$ is threshold if and only if its dual function $f^d(x) = \overline{f(\overline x)}$ is threshold, and that if $(w_1,\ldots, w_n,t)$ is an integral separating structure of $f$, then
$(w_1,\ldots, w_n,\sum_{i = 1}^nw_i-t-1)$ is a separating structure of $f^d$~\cite{CH11}.
Therefore, it suffices to argue that $G$ is total domishold if and only if $f_G^d$ is threshold.

Let $x\in \{0,1\}^{V}$ and let $S(x)$ be the support set of $x$.
By definition, $f_G^d(x) = 0$ if and only if $f(\overline x) = 1$, which is the case if and only if
$V\setminus S$ contains the neighborhood of some vertex.
In other words, $f_G^d(x) = 0$ if and only if $S$ is not a total dominating set.
Hence, if the dual function $f_G^d$ is threshold with an integral separating structure $(w_1,\ldots, w_n,t)$,
then $(w,t+1)$ with $w(v_i) = w_i$ for all $i\in [n]$ is a total domishold structure of $G$, and conversely, if $(w,t)$ is an integral
total domishold structure of $G$, then $(w_1,\ldots, v_n,t-1)$ with $w_i = w(v_i)$
for all $i\in [n]$ is a separating structure of $f_G^d$.

Finally, if
$(w_1,\ldots, w_n,t)$ is an integral separating structure of $f_G$, then
$(w_1,\ldots, w_n,\sum_{i = 1}^nw_i-t-1)$ is a separating structure of $f_G^d$ and hence
$(w;\sum_{i = 1}^nw_i-t)$ with $w(v_i) = w_i$ for all $i\in [n]$
is a total domishold structure of $G$.
\end{proof}
\end{sloppypar}

In view of Proposition~\ref{prop:relation}, TD graphs can also be characterized in terms of the thresholdness of a derived
hypergraph. We now show that the converse is also true, that is, we can test whether a given hypergraph is threshold by testing whether a derived graph is total domishhold. In fact, the derived graph can be assumed to be a split graph. Recall that a graph is {\it split} if its
vertex set can be partitioned into a clique and an independent set.

First, we derive a simple but useful property of threshold hypergraphs. A {\it universal vertex} in a hypergraph ${\cal H}$ is a vertex of ${\cal H}$ contained in all hyperedges of ${\cal H}$.

\begin{lemma}\label{prop:univesal-vertex}
Let ${\cal H = (V,E)}$ be a hypergraph, and let ${\cal H}'= ({\cal V}',{\cal E}')$ be the hypergraph obtained from ${\cal H}$ by adding to it a universal vertex, that is, ${\cal V}' = {\cal V}\cup\{v\}$ and ${\cal E}' = \{e\cup \{v\}\mid e\in {\cal E}\}$.
Then, $\cal H$ is threshold if and only if ${\cal H}'$ is.
\end{lemma}

\begin{proof}
Suppose first that $\cal H$ is threshold, with an integral separating structure $(w,t)$.
Then, the pair $(w',t')$, where
$w':{\cal V}'\to \mathbb{R}_+$ is defined as $$w'(x) = \left\{
                                                            \begin{array}{ll}
                                                              w({\cal V}), & \hbox{if $x = v$;} \\
                                                              w(x), & \hbox{otherwise.}
                                                            \end{array}
                                                          \right.
$$
and $t' = t+w(V)$, is a separating structure of ${\cal H}'$.
Conversely, if ${\cal H}'$ is threshold, with an integral separating structure $(w',t')$, then the pair $(w,t)$, where
$w:{\cal V}\to \mathbb{R}_+$ is the restriction of $w'$ to ${\cal V}$ and $t = t'-w'(v)$, is a separating structure of ${\cal H}'$.
\end{proof}

Given a hypergraph ${\cal H = (V,E)}$, the {\it split-incidence graph} of $\cal H$ (see, e.g.,~\cite{KLMN})
is the split graph $G = {\it SI}({\cal H})$ with $V(G) = {\cal V}\cup {\cal E}'$, where
${\cal E}' = \{e'\mid e\in {\cal E}\}$,
$\cal V$ is a clique, ${\cal E}'$ is an independent set, and
$v\in \cal V$ is adjacent to $e'\in {\cal E}'$ if and only if $v\in e$.

\begin{proposition}\label{prop:hypergraph-to-graph}
For every hypergraph $\cal H$, it holds that
$\cal H$ is threshold if and only if its split-incidence graph ${\it SI}({\cal H})$ is total domishold.
\end{proposition}

\begin{proof}
Let $\cal H=(V,E)$ be a hypergraph. We prove the statement by induction on $|{\cal V}|$. The case $|{\cal V}| = 1$ is trivial since $\cal H$ is threshold, and  ${\it SI}({\cal H})$ is isomorphic to
either  $K_1$ (if ${\cal E} = \emptyset$) or to $K_2$ (if ${\cal E} = \{\emptyset\}$), hence total domishold.

Suppose now that $|{\cal V}|>1$.

If ${\cal H}$ has a universal vertex $v\in {\cal V}$, then let ${\cal H}' = ({\cal V}',{\cal E}')$  be the hypergraph obtained by deleting $v$ from ${\cal H}$, that is, ${\cal V}' = {\cal V}\setminus\{v\}$ and ${\cal E}' = \{e\setminus\{v\}\mid e\in {\cal E}'\}$.
By the inductive hypothesis, ${\cal H}'$ is threshold if and only if its split-incidence graph ${\it SI}({\cal H}')$ is total domishold.
Notice that the split-incidence graph of ${\cal H}'$ is isomorphic to the graph obtained from
the split-incidence graph of ${\cal H}$ by deleting from it a universal vertex (namely $v$).
Hence, Proposition~\ref{prop:dominating-vertex} and Lemma~\ref{prop:univesal-vertex} imply that
$\cal H$ is threshold if and only if its split-incidence graph is total domishold.

We may thus assume that $\cal H$ does not have any universal vertices. Let $G = {\it SI}({\cal H})$ be its split-incidence graph, with ${\cal E}'$ as above.
Let $f_G:\{0,1\}^{V(G)}\to \{0,1\}$ denote the neighborhood function of $G$,
$$f_G(x) = \bigvee_{v\in V(G)} \bigwedge_{u\in N_G(v)}x_u\,,$$
let $f_G':\{0,1\}^{V(G)}\to \{0,1\}$ be the function given by
$$f_G'(x) = \bigvee_{e\in {\cal E}} \bigwedge_{u\in e}x_u\,,$$
and let $f_H$ denote the restriction of $f_G'$ to ${\cal V}$, that is,
$f_{\cal H}:\{0,1\}^{\cal V}\to \{0,1\}$ is given by
$$f_{\cal H}(x) = f_G'(x) = \bigvee_{e\in {\cal E}} \bigwedge_{u\in e}x_u\,.$$

We will show that the following conditions are equivalent, which will imply the proposition:

$(i)$ Hypergraph  $\cal H$ is threshold.

$(ii)$ Function $f_{\cal H}$ is threshold.

$(iii)$ Function $f_G'$ is threshold.

$(iv)$ Function $f_G$ is threshold.

$(v)$ Graph $G$ is total domishold.

The equivalence between $(i)$ and $(ii)$ follows by Proposition~\ref{prop:relation}.
The equivalence between $(ii)$ and $(iii)$ is easy to establish directly from the definition, using the fact that $f_G'$ does not depend on any variable of the form $x_{e'}$ where $e'\in {\cal E}'$ (see~\cite{CH11}).
The equivalence between $(iii)$ and $(iv)$ can be justified as follows.
Since $\cal H$ has no universal vertices, for every vertex $v\in \cal V$ in the clique there exists
a vertex $e'\in {\cal E}'$ in the independent set such that $v\not\in e$.
Hence, $N_G(e') = e\subseteq N_G(v)$.
In particular, this implies that the functions $f_G$ and $f_G'$ are
logically equivalent.
Finally, the equivalence between $(iv)$ and $(v)$ follows from Proposition~\ref{lem:reduction}.
\end{proof}

Proposition~\ref{prop:hypergraph-to-graph} admits a natural converse.
A {\it split partition} of a split graph $G$ is a pair $(K,I)$ such that $K$ is a clique, $I$ is an independent set,
$K\cup I = V(G)$ and $K\cap I = \emptyset$. Given a split graph $G$ with a split partition $(K,I)$, its {\it $I$-neighborhood hypergraph}
is the hypergraph ${\cal IN}(G) = (K,{\cal E})$ where ${\cal E} = \{N_G(v)\mid v\in I\}$. Recall that ${\cal E}$ is a multiset, that is, if two vertices in $I$ have the same neighborhoods in $G$, then we keep in ${\cal E}$ both copies of the neighborhood.

\begin{corollary}\label{cor:graph-to-hypergraph}
For every split graph $G$
with a split partition $(K,I)$, it holds that
$G$ is total domishold if and only if its
$I$-neighborhood hypergraph ${\cal IN}(G)$ is threshold.
\end{corollary}

\begin{proof}
The proposition follows immediately from Proposition~\ref{prop:hypergraph-to-graph} and the observation that
$G$ is isomorphic to the split-incidence graph of its $I$-neighborhood hypergraph ${\cal IN}(G)$.
\end{proof}

We conclude this section with a theorem summarizing several statements equivalent to a fact that a given graph $G$ is total domishold.
To state the theorem , we need two more definitions. To every graph $G= (V,E)$, we associate a split graph $S(G)$
defined as follows:
$V(S(G)) = V\cup W$, where
$$W = \{N_G(v)\mid v\in V~\textrm{and}~(\forall u\in V\setminus\{v\})(N_G(u)\not\subset N_G(v))\}\,,$$
$V$ is a clique, $W$ is an independent set, and
$v\in V$ is adjacent to $X\in W$ if and only if $v\in X$.
(In the above definition of $W$, relation $\subset$ denotes strict set inclusion.)

The following derived hypergraph will also be useful. Given a graph $G$, the {\it reduced neighborhood hypergraph} of $G$ is the hypergraph ${\cal RN}(G) = ({\cal V}, {\cal E})$ where
${\cal V} = V(G)$ and ${\cal E}$ is a set (that is, it does not have any duplicated elements) defined by
$${\cal E} = \{S\mid S\subseteq V(G)\,,~(\exists v\in V(G))(S = N_G(v))~\textrm{and}~(\forall u\in V\setminus\{v\})(N_G(u)\not\subset N_G(v))\}\,.$$

\begin{theorem}\label{prop:equivalences}
For every graph $G$, the following statements are equivalent:
\begin{enumerate}
  \item[$(i)$] Graph $G$ is total domishold.
  \item[$(ii)$] The neighborhood function $f_G$ of $G$ is threshold.
  \item[$(iii)$] The Boolean function $f:\{0,1\}^{V(G)}\to \{0,1\}$ given by the complete DNF $\phi$ of $f_G$ is threshold.
  \item[$(iv)$] The hypergraph ${\cal H}(\phi)$ (where $\phi$ is the complete DNF of $f_G$) is threshold.
  \item[$(v)$] The reduced neighborhood hypergraph ${\cal RN}(G)$ of $G$ is threshold.
  \item[$(vi)$] The split-incidence graph ${\it SI}({\cal RN}(G))$ of ${\cal RN}(G)$ is total domishold.
  \item[$(vii)$] The graph $S(G)$ is total domishold.
\end{enumerate}
\end{theorem}

\begin{proof}
The equivalence between $(i)$ and $(ii)$ was established in Proposition~\ref{lem:reduction}.
The equivalence between $(ii)$ and $(iii)$ is trivial, since $f_G$ is equivalent to the function given by its complete DNF $\phi$.
The equivalence between $(iiv)$ and $(iv)$ follows from Proposition~\ref{prop:relation}.
The equivalence between $(iv)$ and $(v)$ follows from the observation that
reduced neighborhood hypergraphs
 ${\cal H}(\phi)$ and ${\cal RN}(G)$ are isomorphic to each other.
The equivalence between $(v)$ and $(vi)$ follows from Proposition~\ref{prop:hypergraph-to-graph}.
Finally, the equivalence between $(vi)$ and $(vii)$ follows from the observation that
the split-incidence graph ${\it SI}({\cal RN}(G))$ is isomorphic to the split graph $S(G)$.
\end{proof}

\section{Hereditary Total Domishold Graphs}\label{sec:HTD-partial-characterizations}

In this section, we prove our main result: two characterizations of hereditary total domishold graphs.
Our proofs will be based on a new family of hypergraphs defined similarly as the Sperner hypergraphs. Recall that a hypergraph $\cal H = (V,E)$ is said to be {\em Sperner} (or: a {\it clutter}) if no edge of $\cal H$ contains another edge, or, equivalently, if for every two distinct edges $e$ and $f$ of $\cal H$, it holds that $\min\{|e\setminus f|,|f\setminus e|\}\ge 1\,.$
This motivates the following definition.

\begin{definition}\label{def:TD}
A hypergraph $\cal H = (V,E)$ is said to be {\em dually Sperner}
if for every two distinct edges $e$ and $f$ of $\cal H$, it holds that
$$\min\{|e\setminus f|,|f\setminus e|\}\le 1\,.$$
\end{definition}

\begin{lemma}\label{lem:dually-Sperner-threshold}
Every dually Sperner hypergraph is threshold.
\end{lemma}

\begin{proof}
Suppose for a contradiction that there exists a dually Sperner hypergraph $\cal H= (V,E)$ that is not threshold.
By Theorem~\ref{thm:characterization}, there exists an integer $n\ge 2$
and $n$ (not necessarily distinct) subsets $A_1,\ldots, A_n$ of $\cal V$, each containing an edge of $\cal H$,
and $n$ (not necessarily distinct) subsets $B_1,\ldots, B_n$ of $\cal V$, each containing no edge of $\cal H$, such that
for every vertex $v\in \cal V$, condition $(\ref{degree})$ holds.
For every $i\in [n]$, let $e_i$ be an edge of $\cal H$ contained in $A_i$.
Let $i^*\in [n]$ be such that $|e_{i^*}|\le |e_i|$ for all $i\in [n]$.
In particular, this implies that for every $i\in [n]$, it holds that
$|e_{i^*}\setminus e_i|\le |e_i\setminus e_{i^*}|\,,$
which, since $\cal H$ is dually Sperner, implies
\begin{equation}\label{eq1}
|e_{i^*}\setminus e_i|\le 1
\end{equation}
for every $i\in [n]$.
On the other hand, since no $B_i$ contains the edge $e_{i^*}$, we have, for all $i\in [n]$, the inequality
\begin{equation}\label{eq2}
1\le |e_{i^*}\setminus B_i|\,.
\end{equation}
Adding up the inequalities $(\ref{eq2})$ for all $i\in [n]$, we obtain
$n\le \sum_{i\in [n]}|e_{i^*}\setminus B_i|\,.$
This implies the following contradicting chain of equations and inequalities
$$n\le \sum_{i\in [n]}|e_{i^*}\setminus B_i|
= \sum_{i\in [n]}\sum_{v\in e_{i^*}\setminus B_i}1
= \sum_{v\in e_{i^*}}\sum_{i\,:\,v\not\in B_i}1
= \sum_{v\in e_{i^*}}\left(n-|\{i\,:\,v\in B_i\}|\right)$$
$$= \sum_{v\in e_{i^*}}\left(n-|\{i\,:\,v\in A_i\}|\right)
= \sum_{v\in e_{i^*}}\sum_{i\,:\,v\not\in A_i}1
= \sum_{i\in [n]}\sum_{v\in e_{i^*}\setminus A_i}1
= \sum_{i\in [n]}|e_{i^*}\setminus A_i|$$$$
\le \sum_{i\in [n]}|e_{i^*}\setminus e_i|
=  \sum_{{\substack{i\in [n]\\i\neq i^*}}}|e_{i^*}\setminus e_i|
\le   \sum_{{\substack{i\in [n]\\i\neq i^*}}}1
=  n-1\,.$$
The first equality in the second line follows from condition $(\ref{degree})$, while
the first inequality in the third line follows from the fact that $e_i\subseteq A_i$, which implies
$e_{i^*}\setminus A_i\subseteq e_{i^*}\setminus e_i$.
The last inequality follows from (\ref{eq1}).

This contradiction completes the proof.\end{proof}

Let us say that a graph $G$ is {\it $2$-asummable} if its neighborhood function $f_G$ is $2$-asummable, and {\it hereditary $2$-assumable} if
every induced subgraph of it is $2$-asummable. It follows from Theorem~\ref{thm:BF:characterization} and Theorem~\ref{prop:equivalences} that
every total domishold graph is $2$-asummable. Notice that the converse is not true: there exists a $2$-asummable (split) graph $G$ that is not total domishold. This can be easily derived using the results of Section~\ref{sec:TD-graphs-BFs-hypergraphs} and the fact that not every $2$-asummable positive Boolean function is threshold (Theorem 9.15 in~\cite{CH11}). In Theorem~\ref{prop:forbidden-induced-subgraphs} below, we prove several characterizations of HTD graphs, which, in particular, imply that every hereditary $2$-asummable graph is HTD.

\begin{theorem}\label{prop:forbidden-induced-subgraphs}
For every graph $G$, the following are equivalent:
\begin{enumerate}
  \item\label{item1}  $G$ is hereditary total domishold.
  \item\label{item2}  $G$ is hereditary $2$-assumable.
  \item\label{item3}  $G$ is $\{F_1,\ldots, F_{13}\}$-free, where $F_1,\ldots, F_{13}$ are the graphs depicted in Fig.~\ref{fig:forbidden-induced-subgraphs}.
\end{enumerate}
\end{theorem}

\begin{figure}[h!]
  \begin{center}
\includegraphics[width=\linewidth]{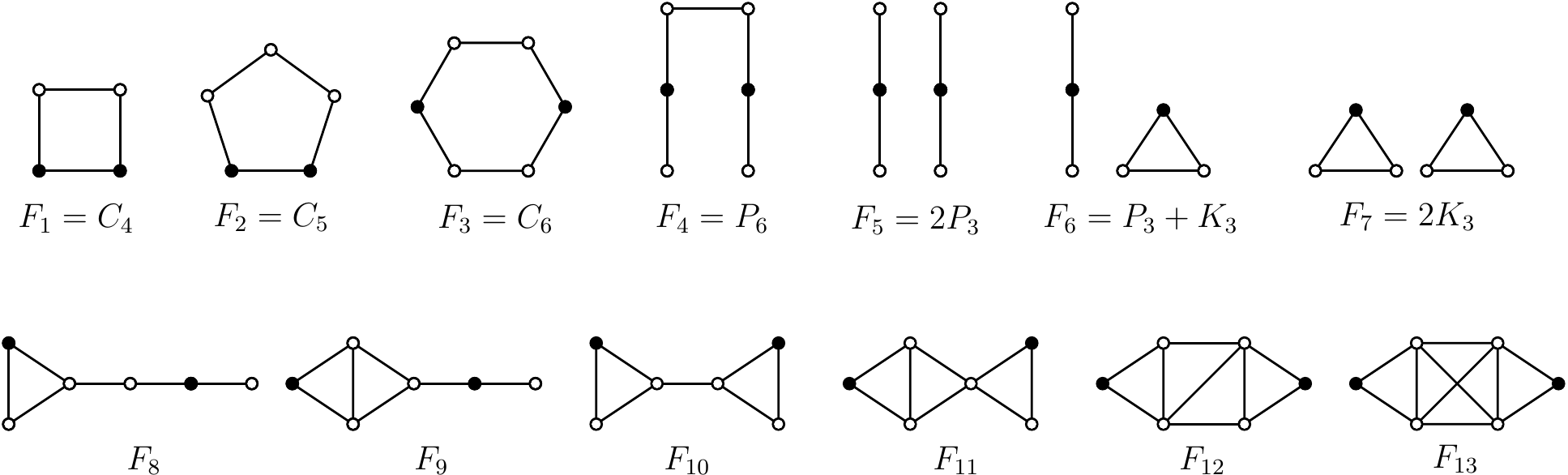}
  \end{center}
  \caption{Graphs $F_1,\ldots, F_{13}$.}
\label{fig:forbidden-induced-subgraphs}
\end{figure}

\begin{proof}
The implication $(\ref{item1})\Rightarrow (\ref{item2})$ follows from Proposition~\ref{lem:reduction} and Theorem~\ref{thm:BF:characterization}.

For the implication $(\ref{item2})\Rightarrow (\ref{item3})$, we only need to verify that none of the graphs $F_1,\ldots, F_{13}$ is $2$-asummable.
Arguing by contradiction, assume that there is a graph $F\in \{F_1,\ldots, F_{13}\}$ that is $2$-asummable.
Take two vertices of degree $2$ in $F$, say $u$ and $v$, such that $u$ and $v$ have disjoint neighborhoods (e.g., the two black vertices in the depiction of $F$ in Fig.~\ref{fig:forbidden-induced-subgraphs}). Denote their respective neighbors by $a,b$ and $c,d$. (If $F\in \{F_1,F_2\}$, then let $c = u$ and $b = v$). It is easy to see that the union of the two neighborhoods $N(u)\cup N(v)$ can be partitioned into two disjoint sets, namely $Y=\{a,c\}$ and $Z=\{b,d\}$, none of which contains the neighborhood of another vertex in the graph.
For a set $S\subseteq V(F)$, let $x^S\in \{0,1\}^{V(F)}$ denote the characteristic vector of $S$, that is,
$$x^S_i = \left\{
           \begin{array}{ll}
             1, & \hbox{if $i\in S$;} \\
             0, & \hbox{otherwise.}
           \end{array}
         \right.$$
Clearly, $x^{N(u)}$ and $x^{N(v)}$ are true points of $f_F$, and $x^Y$ and $x^Z$ are false points of $f_F$.
Since $x^{Y}+x^{Z} = x^{N(u)}+x^{N(v)}$, the neighborhood function $f_F$ is $2$-summable,
a contradiction.

It remains to show the implication $(\ref{item3})\Rightarrow (\ref{item1})$.
Since the class of
$\{F_1,\ldots, F_{13}\}$-free graphs is hereditary, it is enough to show that every
$\{F_1,\ldots, F_{13}\}$-free graph is total domishold.
Let $G$ be an $\{F_1,\ldots, F_{13}\}$-free graph.
By Theorem~\ref{prop:equivalences}, $G$ is total domishold if and only if
its
reduced neighborhood hypergraph ${\cal RN}(G)$ is threshold.
We will prove that ${\cal RN}(G)$ is dually Sperner, and the conclusion will follow from Lemma~\ref{lem:dually-Sperner-threshold}.

Suppose for a contradiction that the reduced neighborhood hypergraph ${\cal RN}(G)$ is not dually Sperner. Then, there exist two hyperedges $e$ and $f$ of
${\cal RN}(G)$ such that $\min\{|e\setminus f|,|f\setminus e|\}\ge 2$.
By the definition of ${\cal RN}(G)$, there exist two vertices $u$ and $v$ of $G$ such that
$e = N_G(u)$ and $f = N_G(v)$.
Moreover, the following condition holds:
\begin{equation}\label{neighborhoods}
\textrm{For $w\in \{u,v\}$, if $x\in V(G)\setminus \{w\}$ then $N_G(x)\not\subset N_G(w)$.}
 \end{equation}

We analyze two cases, depending on whether $u$ and $v$ are adjacent or not.

\medskip

{\it Case 1: $uv\not\in E(G)$.}
Let $a$ and $b$ be two distinct elements in $e\setminus f= N_G(u)\setminus N_G(v)$,  and let $c$ and $d$ be two distinct elements in $f\setminus e = N_G(v)\setminus N_G(u)$. Clearly, vertices $u,v,a,b,c,d$ are all pairwise distinct. Moreover,
$ua,ub,vc,vd\in E(G)$ while $uv,uc,ud,va,vb\not\in E(G)$. Let $k$ be the number of edges in $G$ connecting a vertex in $\{u,a,b\}$ with a vertex of $\{v,c,d\}$. Notice that $k\in \{0,1,2,3,4\}$.

Suppose that $k = 0$. Then, depending on the number of edges in the set $\{ab,cd\}\cap E(G)$, the subgraph of $G$ induced by $\{u,v,a,b,c,d\}$ is isomorphic to either $F_5 = 2P_3$, $F_6 = P_3+K_3$, or $F_7 = 2K_3$, a contradiction.

Suppose that $k = 1$. Then, depending on the number of edges in the set $\{ab,cd\}\cap E(G)$, the subgraph of $G$ induced by $\{u,v,a,b,c,d\}$ is isomorphic to either $F_4 = P_6$, $F_8$, or $F_{10}$, a contradiction.

Suppose that $k = 2$. We may assume without loss of generality that $ac\in E(G)$ and $ad\not\in E(G)$.
Suppose first that $bd\in E(G)$. Then $bc\not\in E(G)$, and depending on the number of edges in the set $\{ab,cd\}\cap E(G)$, the subgraph of $G$ induced by $\{u,v,a,b,c,d\}$ contains either an induced $F_1 = C_4$, $F_2 = C_5$, or $F_3 = C_6$, a contradiction.
Suppose now that $bd\not\in E(G)$. Then $bc\in E(G)$.
Since $G$ is $C_4$-free, $ab\in E(G)$,
 and depending on whether $cd\in E(G)$ or not, the subgraph of $G$ induced by $\{u,v,a,b,c,d\}$ is isomorphic to either $F_{11}$ or $F_9$, a contradiction.

Suppose that $k = 3$. We may assume without loss of generality that $ac, bc, bd\in E(G)$ and $ad\not\in E(G)$.
Since $G$ is $C_4$-free, $ab\in E(G)$ and $cd\in E(G)$. But now, the subgraph of $G$ induced by $\{u,v,a,b,c,d\}$ is isomorphic to $F_{12}$,
a contradiction.

Finally, suppose that $k = 4$. Since $G$ is $C_4$-free, $ab\in E(G)$ and $cd\in E(G)$. But now, the subgraph of $G$ induced by $\{u,v,a,b,c,d\}$ is isomorphic to $F_{13}$, a contradiction.

\medskip
{\it Case 2: $uv\in E(G)$.} Let $a=v$ and $b$ be two distinct elements in $e\setminus f= N_G(u)\setminus N_G(v)$,  and let $c=u$ and $d$ be two distinct elements in $f\setminus e = N_G(v)\setminus N_G(u)$. Since $G$ is $C_4$-free, $bd\not\in E(G)$, and the vertices $\{b,u,v,d\}$ induce a $P_4$ in $G$.

Applying condition~\eqref{neighborhoods} to $w = u$ and $x = d$ implies that $N_G(d)\not\subset N_G(u)$. Hence, there exists a vertex, say $d'$, such that $d'd\in E(G)$ but $d'u\notin E(G)$. Clearly, $d'\not\in \{u,v,b,d\}$.
We claim that $d'$ is not adjacent to $b$. Indeed, if $d'b\in E(G)$, then $G$ contains either an induced $C_4$ (if $d'v\in E(G)$)
or an induced $C_5$ (otherwise), a contradiction.

A symmetric argument shows that $G$ contains a vertex $b'\in V(G)\setminus \{u,v,b,d\}$ that is adjacent to $b$ and non-adjacent to both  $v$ and $d$.
Clearly, $b'\neq d'$.
If $b'd'\in E(G)$, then
depending on the number of edges in the set $\{ub',vd'\}\cap E(G)$, the subgraph of $G$ induced by $\{u,v,a,b,c,d\}$ contains either an induced $F_1 = C_4$, $F_2 = C_5$, or $F_3 = C_6$, a contradiction. If $b'd'\not \in E(G)$, then
depending on the number of edges in the set $\{ub',vd'\}\cap E(G)$, the subgraph of $G$ induced by $\{u,v,a,b,c,d\}$ is isomorphic to either $F_4 = P_6$, $F_8$, or $F_{10}$, a contradiction.
\end{proof}

Theorem~\ref{prop:forbidden-induced-subgraphs} implies a nice structural feature of HTD graphs.
Recall that a graph is said to be {\em $(1,2)$-polar} if it admits a partition of its vertex set into two (possibly empty) parts $K$ and $L$ such that
$K$ is a clique and $L$ induces a subgraph of maximum degree at most $1$.
The following result is an immediate consequence of Theorem~\ref{prop:forbidden-induced-subgraphs} and the forbidden induced subgraph characterization
of $(1,2)$-polar graphs due to Gagarin and Metelski\u{i}~\cite{GM99}.

\begin{corollary}\label{cor:polar-chordal}
Every HTD graph is a $(1,2)$-polar chordal graph.
\end{corollary}

\begin{proof}
In~\cite{GM99}, Gagarin and Metelski\u{i} characterized the set of $(1,2)$-polar graphs with a set of $18$ forbidden induced subgraphs
(see also~\cite{ZZ06}). Only three graphs in this list are chordal: $2P_3$, $P_3+K_3$, and $2K_3$ (that is, the graphs
$F_5, F_6$, and $F_7$ from Fig.~\ref{fig:forbidden-induced-subgraphs}).
This implies that the class of $(1,2)$-polar chordal graphs is exactly the class of
$\{F_5,F_6,F_7,C_4,C_5,C_6,\ldots\}$-free graphs.
Notice that $C_4 = F_1$, $C_5 = F_2$, $C_6 = F_3$, and $F_4 = P_6$ is an induced subgraph of every cycle of order at least $7$.
Therefore, the class of $\{F_1,\ldots, F_{7}\}$-free graphs is a subclass of the class of $(1,2)$-free polar chordal graphs.
In particular, by Theorem~\ref{prop:forbidden-induced-subgraphs}, the same is true for HTD graphs.
\end{proof}

Notice that the converse of Corollary~\ref{cor:polar-chordal} does not hold: graphs $F_8,F_9,\ldots, F_{13}$ are $(1,2)$-polar chordal graphs that are
not TD.

\section{Algorithmic Aspects of TD Graphs}\label{sec:algorithmic-aspects}

As shown in Theorem~\ref{prop:forbidden-induced-subgraphs}, the class of hereditary total domishold graph is characterized by finitely many forbidden induced subgraphs. Therefore, HTD graphs can be recognized in polynomial time.
The next theorem and its proof establish the existence of a polynomial time recognition algorithm also for total domishold graphs,
reducing the problem to the problem of recognizing threshold positive Boolean functions given by a their complete DNF.

\begin{theorem}\label{thm:poly-time-recognition}
There exists a polynomial time algorithm for recognizing total domishold graphs. If the input graph $G$ is total domishold, the algorithm also computes an integral total domishold structure of $G$.
\end{theorem}

\begin{proof}
Theorem~\ref{thm:threshold-Bfs} and Proposition~\ref{lem:reduction} imply that the following polynomial time
algorithm correctly determines whether $G$ is total domishold, and if this is the case,
computes a total domishold structure of it.
First, compute the complete DNF $\phi$ of the neighborhood function $f_G$ of $G$.
More specifically, let
$\phi = \bigvee_{S\in {\cal N}} \bigwedge_{u\in S}x_u$
where ${\cal N}$ is the set of neighborhoods of vertices of $G$ that do not properly contain any other neighborhood.
Second, apply the algorithm given by Theorem~\ref{thm:threshold-Bfs} to the input $\phi$.
If the algorithm detects that $f_G$ is not threshold, then $G$ is not total domishold. Otherwise, the algorithm will compute
an integral separating structure
$(w_1,\ldots, w_n,t)$  of $f_G$, in which case Proposition~\ref{lem:reduction} implies that
$(w;\sum_{i = 1}^nw_i-t)$ with $w(v_i) = w_i$ for all $i\in [n]$
is a total domishold structure of $G$.
\end{proof}

In view of Proposition~\ref{prop:relation}, Theorem~\ref{thm:threshold-Bfs} also implies the existence of a polynomial time algorithm for recognizing threshold hypergraphs. Recall that the known algorithm is based on dualization and linear programming; the existence of a {purely combinatorial} polynomial time algorithm for recognizing threshold Boolean functions/hypergraphs is an open problem~\cite{CH11}. In view of the reduction from the proof of Theorem~\ref{thm:poly-time-recognition}, it is a natural question whether there exists a purely combinatorial polynomial time algorithm for recognizing total domishold graphs.
Proposition~\ref{prop:hypergraph-to-graph} and Corollary~\ref{cor:graph-to-hypergraph} imply the following:
\begin{proposition}
There exists a purely combinatorial polynomial time algorithm for recognizing total domishold split graphs if and only if
there exists a purely combinatorial polynomial time algorithm for recognizing threshold hypergraphs (equivalently: for recognizing threshold positive Boolean functions given by their complete DNF).
\end{proposition}

Let us now examine some consequences of Theorem~\ref{thm:poly-time-recognition}. The {\it total domination number} $\gamma_t(G)$ of a graph $G$ without isolated vertices is the minimum size of a total dominating set in $G$, and the {\it total dominating set problem} is the problem of computing the total domination number of a given graph without isolated vertices.
The problem is \NP-hard in general, and remains hard even for restricted graph classes such as bipartite graphs or split graphs~\cite{CS90}. On the positive side, polynomial time algorithms have been designed for several graph classes~\cite{Kratsch,Henning-Yeo}. With the exception of dually chordal graphs~\cite{BCD} and DDP-graphs~\cite{Henning-Yeo}, all known polynomial time algorithms for the total dominating set problem we are aware of deal with hereditary graph classes. The following result provides another example of a non-hereditary graph class for which the problem is polynomial.

\begin{proposition}\label{prop:MTD-problem-polynomial}
The total dominating set problem is solvable in polynomial time for total domishold graphs.
\end{proposition}

\begin{proof}
Applying Theorem~\ref{thm:poly-time-recognition}, we may assume that the input graph $G$ is given together with an integral total domishold structure $(w,t)$.
The following greedy approach can now be used to find a minimum-sized total dominating set $S$:
\begin{enumerate}
  \item Sort the vertices according to their weights so that $w(v_1)\ge w(v_2)\ge \cdots \ge w(v_n)$.
  \item Start with the empty set, $S=\emptyset$, and, as long as $w(S)<t$, keep adding to $S$ vertices in the above order $(v_1,\ldots, v_n)$.
 \item Return $S$.
\end{enumerate}
The correctness and the polynomial running time of this algorithm are immediate.
\end{proof}

While the total dominating set problem is \NP-hard for chordal graphs (in fact, even for split graphs~\cite{CS90}),
Proposition~\ref{prop:MTD-problem-polynomial} shows that the problem is polynomial in the class of HTD graphs, which, by
Corollary~\ref{cor:polar-chordal}, is a subclass of chordal graphs.

We now examine some consequences of Proposition~\ref{prop:MTD-problem-polynomial} for the dominating set problem.
The {\it domination number} $\gamma(G)$ of a given graph $G$ is the minimum size of a dominating set in $G$.
The {\it dominating set problem} is the problem of computing the domination number of a given graph.
For general graphs, it is $\NP$-hard to compute the domination number~\cite{GJ79}.
Moreover, for every $\epsilon >0$, there is no polynomial time algorithm approximating
the domination number for graphs without isolated vertices within a factor of \hbox{$(1-\epsilon)\ln n$}, unless
$\NP\subseteq {\textrm{DTIME}}(n^{O(\log\log n)})$~\cite{Feige-98,CC04}, and it is $\NP$-hard to approximate the domination number
within a factor of \hbox{$0.2267\ln n$} (see~\cite{CMV13}).

\begin{proposition}\label{prop:2-approx}
There exists a $2$-approximation algorithm for the dominating set problem in the class of total domishold graphs.
\end{proposition}

\begin{proof}
Given a TD graph $G$, let $S$ be the set of isolated vertices in $G$, and let $G' = G-S$.
We may assume that $G$ has at least one edge (otherwise we have $\gamma(G) = |V(G)|$),
and hence $G'$ is non-empty graph with a total dominating set.
By Proposition~\ref{prop:MTD-problem-polynomial},
we can compute a minimum total dominating set $D$ in $G'$ in polynomial time.
Let $D' = D\cup S$.

Clearly, $D'$ is a dominating set in $G$.
On the other hand, denoting by $D^*$ a minimum dominating set in $G$,
we see that $D^*-S$ is a minimum dominating set in the graph $G'$.

It is easy to see that for every graph $H$ without isolated vertices,
its domination and total domination numbers are related as follows:
$\gamma(H)\le \gamma_t(H)\le 2\gamma(H)$.
In particular,
$\gamma_t(G')\le 2\gamma(G')$, which  implies
$|D| = \gamma_t(G')\le 2\gamma(G') = 2|D^*-S|$
and consequently
$$|D'| = |D\cup S| = |D|+|S| \le 2|D^*-S|+|S|\le 2|D^*|\,.$$
Therefore, the algorithm that computes and outputs $D'$ is a $2$-approximation algorithm for the domination problem on total domishold graphs.
\end{proof}

We leave it as an open problem to determine the exact complexity status of the dominating set problem in the classes of TD and HTD graphs.

\begin{remark}
Simple reductions from general graphs show that the problems of computing the independence number, the clique number and the chromatic number of a given TD graph are \NP-hard, and also $\NP$-hard to approximate to within a factor of $n^{1-\epsilon}$. (See Appendix~\ref{hardness}.)
On the other hand, since HTD graphs are chordal, all these problems are polynomial for HTD graphs.
\end{remark}

\section{Conclusion}\label{sec:conclusion}

In conclusion, we mention some open problems related to total domishold graphs.

\begin{problem}
Is there a purely combinatorial polynomial time algorithm for recognizing total domishold graphs?
For recognizing split total domishold graphs?
\end{problem}
An answer to either of these questions would resolve the open problem
regarding the existence of a {purely combinatorial} polynomial time algorithm for recognizing threshold positive Boolean functions given by a complete DNF~\cite{CH11}.

As showed in this paper, the class of HTD graphs forms a proper subclass of $(1,2)$-polar chordal graphs and a proper superclass of the complements of domishold graphs.
A more detailed structural analysis of HTD graphs seems to
deserve further attention. For instance, Theorem~\ref{prop:forbidden-induced-subgraphs} implies that a split graph $G$ is HTD if and only if it is $F_{13}$-free. Moreover, in~a companion paper~\cite{CM13} we prove that by
replacing in the list of forbidden induced subgraphs of HTD graphs the graphs $F_{11}$ and $F_{12}$
from Fig.~\ref{fig:forbidden-induced-subgraphs} with two of their respective induced subgraphs on $5$ vertices,
we obtain a class of graphs every connected member of which arises from a threshold graph by
attaching an arbitrary (non-negative) number of leaves to each of its vertices.
In particular, this structure implies a linear time recognition algorithm for graphs in this class.

Further insight into the structure of general HTD graphs might imply the existence of a more
efficient recognition algorithm of HTD graphs
than the
one with running time $O(|V(G)|^6)$ implied by Theorem~\ref{prop:forbidden-induced-subgraphs}.

\begin{problem}
Is there a linear time algorithm for recognizing hereditary total domishold graphs?
\end{problem}

We conclude with a question motivated by the existence of a
$2$-approximation algorithm for the dominating set problem in the class of total domishold graphs (Proposition~\ref{prop:2-approx}).

\begin{problem}
Determine the computational complexity of the dominating set problem on total domishold  / hereditary total domishold graphs.
\end{problem}

\subsection*{Acknowledgements}

We would like to thank Endre Boros for stimulating discussions and helpful comments.

\bibliographystyle{abbrv}

\appendix

\section{Hardness Proofs}\label{hardness}

Here we provide the hardness proofs establishing the computational complexity of computing the independence, the clique and the chromatic numbers of a given total domishold graph. Recall that the independence number of a graph $G$ is denoted by $\alpha(G)$ and equals the maximum size of an independent set in $G$, the clique number of a graph $G$ is denoted by $\omega(G)$ and equals the maximum size of a clique in $G$, and the chromatic number of a graph $G$ is denoted by $\chi(G)$ and equals the minimum integer $k$ such that $V(G)$ can be partitioned into $k$ independent sets.
For general graphs, all these three parameters are $\NP$-hard to compute, and also $\NP$-hard to approximate to within a factor of $n^{1-\epsilon}$, for every $\epsilon >0$~\cite{Zuckerman}.

\begin{proposition}
For every $\epsilon >0$, it is $\NP$-hard to approximate the independence number of a given TD graph to within a factor of $n^{1-\epsilon}$.
\end{proposition}

\begin{proof}
We reduce from the independent set problem in general graphs.
Given a graph $G$, we construct a graph $G'$ from $G$ by adding to it first an isolated, and then a universal vertex.
By Propositions~\ref{prop:dominating-vertex} and~\ref{prop:isolated-vertex}, the obtained graph is TD.
It is straightforward to verify that the independence numbers of $G$ and $G'$ are related by the equation
$\alpha(G') = \alpha(G)+1$. The result follows.
\end{proof}

\begin{proposition}
For every $\epsilon >0$, it is $\NP$-hard to approximate the clique number of a given TD graph to within a factor of $n^{1-\epsilon}$.
\end{proposition}

\begin{proof}
We reduce from the clique problem in general graphs.
Given a graph $G$ that is the input to the clique problem, we may assume without loss of generality that $G$ has at least one edge, and that it has no isolated vertices.
We construct a graph $G'$ from $G$ by adding a private neighbor to each vertex of $G$.
The obtained graph $G'$ has a unique minimal TD set, namely $V(G)$, and hence
setting $t = |V(G)|$ and defining $w:V(G') \to \mathbb{R}_+$ as
$$w'(x) = \left\{
            \begin{array}{ll}
              1, & \hbox{if $x\in V(G)$;} \\
              0, & \hbox{otherwise.}
            \end{array}
          \right.
$$
produces a total domishold structure of $G'$.
Hence $G'$ is TD. Furthermore, we clearly have $\omega(G') = \max\{2,\omega(G)\} = \omega(G)$, and the result follows.
\end{proof}

\begin{proposition}
For every $\epsilon >0$, it is $\NP$-hard to approximate the chromatic number of a given TD graph to within a factor of $n^{1-\epsilon}$.
\end{proposition}

\begin{proof}
We reduce from the chromatic number problem in general graphs.

Given a graph $G$ that is the input to the chromatic number problem,
we may assume without loss of generality that $G$ has no isolated vertices.
We construct a graph $G'$ from $G$ by adding a private neighbor to each vertex of $G$.
As above, the obtained graph $G'$ is TD. Furthermore, we have $\chi(G') = \max\{2,\chi(G)\} = \chi(G)$, and the result follows.
\end{proof}

\end{document}